\documentclass[12pt]{article}
\usepackage[T2A]{fontenc}
\usepackage[cp1251]{inputenc}
\usepackage[tbtags]{amsmath}
\usepackage{amsfonts,amsmath,amssymb,mathrsfs,amscd}
\usepackage{amsthm}
\usepackage{verbatim}
\usepackage{graphicx}
\usepackage{xcolor}
\usepackage{enumerate}
\usepackage{textcomp}
\usepackage{latexsym}
\usepackage{verbatim}
\usepackage[mathscr]{eucal}

\theoremstyle{plain} 
\newtheorem{theorem}{Theorem}

\theoremstyle{definition}
\newtheorem{definition}{Definition}
\newtheorem{remark}{Remark}
\newtheorem{example}{Example}

\usepackage[pdftex,unicode,colorlinks,linkcolor=blue,citecolor=red,bookmarksopen,
pdfhighlight=/N]{hyperref}

\newcommand{\RR}{\mathbb{R}} 

\newcommand{\NN}{\mathbb{N}} 
\newcommand{\ZZ}{\mathbb{Z}}
\newcommand{\DD}{\mathbb{D}} 

\newcommand{\BB}{\mathbb{B}}  
 
\newcommand{\const}{{\rm const}}

\newcommand{\dd}{\,{\rm d}}
\DeclareMathOperator{\inhull}{\rm hull-in}

\DeclareMathOperator{\diam}{diam}
\DeclareMathOperator{\dist}{dist} 
 
\DeclareMathOperator{\clos}{clos} 
\DeclareMathOperator{\Int}{int}
\DeclareMathOperator{\Meas}{Meas}

\DeclareMathOperator{\har}{har}

\DeclareMathOperator{\Har}{har}
\DeclareMathOperator{\comp}{c}

\DeclareMathOperator{\sbh}{sbh}

\DeclareMathOperator{\supp}{supp} 
\DeclareMathOperator{\loc}{loc}

\DeclareMathOperator{\Borel}{Borel}
\DeclareMathOperator{\dom}{dom}
\DeclareMathOperator{\Dom}{Dom}

\DeclareMathOperator{\conn}{conn}
\DeclareMathOperator{\Conn}{Conn}

\title{Classical Balayage of Charges and Measures}
\author{\bf Bulat N. Khabibullin\footnote{This study was financially  supported by the Russian Science Foundation (projects No.~18-01-00002.)}
 \and  \bf Enzhe Menshikova}

\begin{document}

\maketitle

\begin{abstract}
We investigate some properties of balayage of charges and measures for subclasses of subharmonic functions and their relationship to the geometry of  domain or open set in 
finite-dimensional Euclidean space  where this balayage is considered.



\end{abstract}

\section{Introduction}\label{int}

We have are considered in the survey \cite{KhaRozKha19} various general concepts of  balayage. In this article we deal with a particular case of such balayage with respect to special  classes of  test subharmonic  functions. 
The general concept of balayage can be defined as follows. Let $F$ be a set and $(R,\leq)$ be a   (pre-)ordered set with (pre-)order relation $\leq$. A function $f\colon F\to R$
can be called a {\it  balayage} of a function $g\in F\to R$ {\it  for a subset\/} $\mathcal V \subset F$, and we write $g\preceq_{\mathcal V}  f$,  if the function $f$ majorizes the function $g$ on $\mathcal V $:
\begin{equation}\label{b0}
g(v)\leq f(v) \quad \text{for all $v\in \mathcal V$} .
\end{equation}

In this article, we use the balayage when $F$ is a class of integrals  defined by charges or positive measures on a subdomain $D$ of  finite-dimensional Euclidean space, and classes  $\mathcal V$ are  special subclasses  of subharmonic functions on $D$. In this case, relation \eqref{b0} turns into inequalities of the form 
\begin{equation}\label{b0mu}
g(v):=\int_D v \dd \vartheta\leq \int_D v \dd \mu=:f(v) \quad \text{for all $v\in \mathcal V$}, 
\end{equation}
where $\vartheta$ and  $\mu$ is a pair of charges or measures.

The main and some special properties of charges and measures are outlined in Sec. \ref{Ssec_balm}, Theorems \ref{Prtr}--\ref{Pr_pol}, Examples \ref{sbhJ}--\ref{5}.

\section{Definitions, notations and conventions}\label{Ss12}
\setcounter{equation}{0}

The reader can skip this Section \ref{Ss12}
and return to it only if necessary.

\subsection{\bf Sets, order, topology}\label{SsSets} 

As usual, $\mathbb N:=\{1,2, \dots\}$,  $\mathbb R$ and $\mathbb C$ are the sets 
of all {\it natural, real \/} and {\it complex\/} numbers, respectively; 
$\NN_0:=\{0\}\cup\NN$   is French natural series, and $\ZZ:=\NN_0\cup\NN_0$. 

For $d \in \NN$ we  denote by $\mathbb R^d$ the {\it $d$-dimensional real  Euclidean  space\/} with the standard {\it Euclidean norm\/} $|x|:=\sqrt{x_1^2+\dots+x_d^2}$ for $x=(x_1,\dots ,x_d)\in \RR^d$
and the distance function $\dist (\cdot, \cdot)$.
For  the {\it real line\/} $\RR=\RR^1$ with  {\it Euclidean norm-module\/} $|\cdot |$,   
\begin{subequations}\label{df:R}
\begin{align}
\RR_{-\infty}:=\{-\infty\}\cup \RR,\; 	\RR_{+\infty}:=\RR\cup 
\{+\infty\}, \; &|\pm\infty|:=+\infty; 
\;  \RR_{\pm\infty}:=\RR_{-\infty}\cup \RR_{+\infty}
\tag{\ref{df:R}$_\infty$}\label{df:Rr}\\
\intertext{is {\it extended real line\/} in the end topology with two ends $\pm \infty$, with   the order relation $\leq$ on $\RR$ complemented by the 
inequalities $-\infty \leq x\leq +\infty$ for $x\in \RR_{\pm\infty}$, with the {\it positive real axis}}
\RR^+:= \{x\in \RR\colon x\geq 0\}, \; \RR_{+\infty}^+:=\RR^+\cup\{+\infty\}, &\; \begin{cases}
x^+&:=\max\{0,x \},\\
 x^-&:=(-x)^+,
\end{cases}
\; \text{for $x\in \RR_{\pm\infty}$},
\tag{\ref{df:R}$^+$}\label{df:R+}
\\
S^+:=\{x\geq 0\colon x\in S \}, \quad S_*:=S\setminus \{0 \} &\quad\text{for $S\subset \RR_{\pm\infty}$}, \quad \RR_*^+:=(\RR^+)_*, 
\tag{\ref{df:R}$_*^+$}\label{df:R*}\\
x\cdot (\pm\infty):=\pm\infty=:(-x)\cdot (\mp\infty)& \quad \text{for $x\in \RR_*^+\cup (+\infty)$}, 
\tag{\ref{df:R}$_\pm$}\label{{infty}+}\\
\frac{x}{\pm\infty}:=0\quad\text{for $x\in  \RR$},&\quad  \text{but $0\cdot (\pm\infty):=0$}
\tag{\ref{df:R}$_0$}\label{{infty}0}
\end{align}
\end{subequations} 
unless otherwise specified. An open connected (sub-)set of $\RR_{\pm\infty}$  is a {\it  (sub-)interval}  of $\RR_{\pm\infty}$.
The {\it  Alexandroff\/} one-point {\it compactification\/} of $\mathbb R^d$ is denoted by $\mathbb R^d_{\infty}:=\mathbb R^d \cup \{\infty\}$.

The same symbol $0$ is used, depending on the context, to denote the number zero, the origin, zero vector, zero function, zero measure, etc. The {\it positiveness\/} is everywhere understood as $\geq  0$ according to the context.
Given $x\in \RR^d$ and\footnote{A reference mark over a symbol of (in)equality, inclusion, or more general binary relation, etc. means that this relation is somehow related to this reference.}  $r\overset{\eqref{df:R+}}{\in} \RR_{+\infty}^+$, we set 
\begin{subequations}\label{B}
\begin{align}
B(x,r):=\{x'\in \RR^d \colon |x'-x|<r\},&
\quad \overline{B}(x,r):=\{x'\in \RR^d \colon |x'-x|\leq r\},
\tag{\ref{B}B}\label{{B}B}
\\
\quad B(\infty,r):=\{x\in \RR_{\infty}^d \colon |x|>1/r\},&\quad 
\overline B(\infty,r):=\{x\in \RR_{\infty}^d \colon |x|\geq 1/r\},
\tag{\ref{B}$_\infty$}\label{{B}infty}
\\
B(r):=B(0,r),\quad \BB:=B(0,1),& \quad \overline{B}(r):=\overline{B}(0,r),
\quad \overline \BB:=\overline B(0,1).
\tag{\ref{B}$_1$}\label{{B}1}
\\
B_{\circ}(x,r):=B(x,r)\setminus \{x\} ,&\quad 
 \overline{B}_{\circ}(x,r):=\overline{B}(x,r)\setminus \{x\}.
\tag{\ref{B}$_\circ$}\label{Bo}
\end{align}
\end{subequations} 
Thus, the basis of open (respectively closed) neighborhood of the point $x \in \RR_{\infty}^d$ is {\it open\/} (respectively {\it closed\/}) {\it balls\/} 
$B(x,r)$ (respectively $\overline B(x,r)$) centered at  $x$ with radius $r>0$.

Given a subset $S$ of $\RR^d_{\infty}$, the \textit{closure\/} 
$\clos S$, the\textit{ interior\/} $\Int S$  and the \textit{boundary\/} $\partial S$ will always be taken relative $\RR^d_{\infty}$. For $S'\subset S\subset \RR^d_{\infty}$ we write  
$S'\Subset S$ if $\clos S'\subset \Int S$.   
An open connected (sub-)set of $\RR^d_{\infty}$  is a {\it  (sub-)domain}  of $\RR^d_{\infty}$. 

\subsection{\bf Functions.}\label{Functions} Let $X,Y$ are sets. We denote by $Y^X$ the set of all functions  $f\colon X\to Y$. The value $f(x) \in Y$ of an arbitrary function $f\in X^Y$ is not necessarily defined for all $x \in X$. The restriction of a function f to $S \subset X$ is denoted by $f\bigm|_{S}$. We set 
\begin{equation}\label{RX}
\RR_{-\infty}^X\overset{\eqref{df:Rr}}{:=}(\RR_{-\infty})^X, \quad
\RR_{+\infty}^X\overset{\eqref{df:Rr}}{:=}(\RR_{+\infty})^X,\quad
\RR_{\pm\infty}^X\overset{\eqref{df:Rr}}{:=}(\RR_{\pm\infty})^X.  
\end{equation}
A function $f\in \RR_{\pm\infty}^X$ is said to be {\it extended numerical.\/}  
For extended numerical functions $f$, we set 
\begin{equation}\label{dom}
\begin{split}
\Dom_{-\infty}:=f^{-1}(\RR_{-\infty})\subset X,& 
\quad  \Dom_{+\infty} f:=f^{-1}(\RR_{+\infty})\subset X, \\
  \Dom f:=f^{-1}(\RR_{\pm\infty})&=
\Dom_{-\infty} f\bigcup \Dom_{+\infty}f\subset X,\\
\dom f:=f^{-1}(\RR)=&\Dom_{-\infty} f\bigcap \Dom_{+\infty}f\subset X, 
\end{split}
\end{equation} 
For $f,g\in \RR_{\pm\infty}^X$  we write $f= g$ if  
$\Dom f=\Dom g=:D$ and $f(x)=g(x)$ for all $x\in D$, 
and we write $f\leq g$ if $f(x)\leq g(x)$ for all $x\in D$.
For $f\in \RR_{\pm\infty}^X$, $g\in \RR_{\pm\infty}^Y$  and a set $S$, we write 
``$f = g$ {\it on\/} $S$\,'' or  ``$f \leq g$ {\it on\/} $S$\,'' if 
 $f\bigm|_{S\cap D}= g\bigm|_{S\cap D}$ or $f\bigm|_{S\cap D}\leq g\bigm|_{S\cap D}$ respectively.

For $f\in F\subset \RR_{\pm\infty}^X $, we set $f^+\colon x\mapsto \max \{0,f(x)\}$,
$x\in \Dom f$,  $F^+:=\{f\geq 0 \colon f\in F\}$. So, $f$  is \textit{positive\/} on $X$ if $f=f^+$, and  we write ``$f\geq 0$ {\it on\/} $X$''.

 For topological space $X$, $C(X)\subset \RR^X$  is the vector space over $\RR$
of all continuous functions. We denote the function
identically equal to resp. $-\infty$ or $+\infty$ on a set 
by the symbols $\boldsymbol{-\infty}$ or $\boldsymbol{+\infty}$.

For an open set  $O\subset \RR^d_{\infty}$,
 we denote  by $\har (O)$ and  $\sbh (O)$  the classes  of all {\it harmonic\/} (locally affine for m = 1) and    {\it subharmonic\/} (locally convex for $m = 1$) functions on $O$, respectively.  The class $\sbh ( O)$  contains the {\it minus-infinity function\/} 
$-\infty$; 
\begin{equation}\label{sbh}
 \sbh_*(  O):=\sbh\,(  O)\setminus 
\{\boldsymbol{-\infty}\}, \quad 
	\sbh^+(  O):=( \sbh (  O))^+.
\end{equation}
 If $o\notin O\ni \infty$, then we can to use the  \textit{inversion\/} in  the 
sphere $\partial B(o,1)$ centered at $o \in  \RR^d$:
\begin{subequations}\label{stK}
\begin{align}
\star_{o} \colon x\longmapsto x^{\star_{o}}&:= \begin{cases}
o\quad&\text{for $x=\infty$},\\
o+\frac{1}{|x-o|^2}\,(x-o)\quad&\text{for $x\neq o,\infty$},\\
\infty\quad&\text{for $x=o$},
\end{cases}
\qquad \star:=\star_0=:\star_{\infty}
\tag{\ref{stK}$\star$}\label{stK*}
\\
\intertext{together with the  {\it Kelvin transform\/} \cite[Ch. 2, 6; Ch. 9]{Helms}
}
u^{\star_o}(x^{\star_o})&=|x-o|^{d-2}u(x), \quad x ^{\star_o}\in   
O^{\star_o}:=\{x^{\star_o}\colon x\in
  O\}, 
\tag{\ref{stK}u}\label{stKu}
\\
&\Bigl(u\in \sbh (O)\Bigr)\Longleftrightarrow  \Bigl(u^{\star_o}\in \sbh (O^{\star_o})\Bigr).
\tag{\ref{stK}s}\label{stKs}
\end{align}
\end{subequations}

For a  subset $S\subset \RR_{\infty}^d$,  the classes $\Har (S)$, $\sbh(S)$, and $C^k(S)$ with $k\in \NN\cup \{\infty\}$ consist of the restrictions to $S$ of {\it harmonic, subharmonic, {\rm  and}
  k times continuously differentiable functions\/} in some (in general, its own for each function) open set $O\subset \RR_{\infty}^d$ containing $S$.   A class 
$\sbh_*(S)$ is  defined like previous class \eqref{sbh},
\begin{equation}\label{+S}
\sbh^+(S)\overset{\eqref{sbh}}{:=}\bigl\{u\in \sbh(S)\colon u\geq 0 \text{ on } S\bigr\}.
\end{equation}

By $\const_{a_1,a_2,\dots}\in \RR$ we denote constants, and constant functions, in general, depend on $a_1,a_2,\dots$ and, unless otherwise specified, only on them,
where the dependence on dimension $d$ of $\RR_{\infty}^d$ will be not specified and not discussed; $\const^+_{\dots}\geq 0$.  

\subsection{\bf Measures and charges.} 

Let $\Borel (S)$ be the class of all Borel subsets in $S\in \Borel(\RR_{\infty}^d)$. We denote by $\Meas(S)$  the class of all Borel signed measures, or, \textit{charges\/} on $S\in{\Borel} (\RR_{\infty}^d)$;   $\Meas_{\comp}(S)$ is the class of charges $\mu \in \Meas(S)$ with a compact support $\supp \mu \Subset S$; 

\begin{subequations}\label{m}
\begin{align}
\Meas^+(S)&:=\{ \mu\in  \Meas (S)\colon \mu\geq 0\},
 \; \Meas_{\comp}^+(S):= \Meas_{\comp} (S)\cap \Meas^+(S);
\tag{\ref{m}$^+$}\label{m+}\\ 
  \Meas^{1+}(S)&:=\{\mu \in \Meas^+(S)\colon \mu(S)=1 \} \text{, \;  \it  probability measures}.
\tag{\ref{m}$^1$}\label{m1}
\end{align}
\end{subequations}
For a charge $\mu \in \Meas(S)$, we let
$\mu^+$, $\mu^-:=(-\mu)^+$ and $|\mu| := \mu^+ +\mu^-$ respectively denote its {\it upper, lower,\/} and {\it total variations.}  So, $\delta_x \in \Meas_{\comp}^{1+} (S)$
is the {\it Dirac measure\/} at a point $x \in S$, i.e., $\supp \delta_x = \{x\}$, $\delta_x (\{x\}) = 1$. We denote by  $\mu\bigm|_{S'}$
the restriction of $\mu$ to  $S'\in {\Borel} (\RR_{\infty}^d)$.

If the Kelvin transform \eqref{stK} translates the subharmonic function $u$ into another function $u^{\star}_o$ \eqref{stKu}, then its Riesz measure $\upsilon$ is transformed  common use  image under its own mapping-inversion of type $1$ or $2$. These rules are described in detail in L. Schwartz's  monograph \cite[Vol.~I,Ch.IV, \S~6]{Schwartz} and we do not dwell on them here, although here interesting questions arise, for example, for the Bernstein\,--\,Paley\,--\,Wiener\,--Mary Cartwright classes of entire functions \cite{Havin}, \cite{Koosis96}, \cite{BaiTalKha}, \cite{KhaTalKha14} etc.

Given $S\in{\Borel}(\RR_{\infty}^d)$ and $\mu\in \Meas (S)$,  the class $L^1_{\loc} (S, \mu)$ consists of all extended numerical locally integrable functions with respect to the measure $\mu$ on $S$; $L^1_{\loc} (S):=L^1_{\loc} (S, \lambda_d)$.
For $ L\subset L^1_{\loc}(S,\mu )$, we define a subclass
\begin{equation}\label{MLl}
L \dd \mu:=\bigl\{\nu \in  \Meas (S)\colon 
\text{\it  there exists $g\in  L$ such that\/   $\dd \nu=g \dd \mu$} \bigr\}
\end{equation} 
 of the class of all absolutely continuous charges with respect to $\mu$. 
For $\mu \in \Meas(S)$, we set  
\begin{equation}\label{mB}
\mu(x,r):=\mu\bigl(B(x,r)\bigr) \text{ if $B(x,r)\overset{\eqref{B}}{\subset} S$}.
\end{equation}
Let ${\bigtriangleup}$  be the  the {\it Laplace operator\/}  acting in the sense of the
theory of distributions, $\Gamma$ be the \textit{gamma function,}
\begin{equation}\label{sd-1}
s_{d-1}:=\frac{2\pi^{d/2}}{\Gamma (d/2)}
\end{equation}
be the \textit{surface area\/} of the \textit{$(d-1)$-dimensional unit sphere\/} $\partial \BB$ embedded in $\RR^d$.
For function $u\in \sbh_*(O)$, the {\it  Riesz measure of\/} $u$ is a  Borel 
(or Radon \cite[A.3]{R}) \textit{ positive measure }
\begin{equation}\label{df:cm}
\varDelta_u:= c_d {\bigtriangleup}  u\in \Meas^+(  O),  \quad 
c_d\overset{\eqref{sd-1}}{:=}\frac{1}{s_{d-1}(1+( d-3)^+)}=\frac{\Gamma(d/2)}{2\pi^{d/2}\max \{1, d-2\bigr\}}.
\end{equation}
In particular,   $\Delta_u(S)<+\infty$ for each subset $S\Subset   O$.
By definition, $\Delta_{\boldsymbol{-\infty}}(S):=+\infty$ for all $S\subset 
  O$. 

We use different variants of \textit{outer Hausdorff $p$-measure} $\varkappa_p$ with $p\in \NN_0$:
\begin{subequations}\label{df:sp}
  \begin{align}
\varkappa_p(S)&:=b_p \lim_{0<r\to 0} \inf \biggl\{\sum_{j\in \NN}r_j^p\,\colon  
  S\subset \bigcup_{j\in \NN}B(x_j,r_j),  0\leq r_j<r\biggr\}, 
\tag{\ref{df:sp}H}\label{df:spH}
 \\ 
b_p &\overset{\eqref{df:cm}}{:=}
\begin{cases} 0\quad&\text{if  $p=0$,}\\
\dfrac{s_{p-1}}{p}\quad&\text{ if $p\in \NN,$}
\end{cases}  \quad \text{is the {\it volume of the unit ball $\BB$ in  $\RR^p$}}.
\tag{\ref{df:sp}b}\label{df:spb}
\end{align}
\end{subequations}
Thus, for $p=0$, for any $S\subset \RR^d$, its Hausdorff $0$-measure $\varkappa_0(S)$ is to the cardinality $\#S$ of $S$, for $p=d$ we see that $\varkappa_d\overset{\eqref{df:spH}}{=:}\lambda_d$ is  
 the {\it Lebesgue measure\/} to Borel proper subsets
$S \subset \RR_{\infty}^d$, where, if $\infty \in S$, we preliminary use the inversion\eqref {stKu}, and $\sigma_{d-1}:=\varkappa_{d-1}\bigm|_{\partial \BB}$ is the $(d-1)$-dimensional  surface measure of area on the unit sphere $\partial \BB$ in the usual sense. 

\subsection{\bf Topological concepts. Inward-filled hull of set in  open set}\label{hullin} 

Let $O$ be a topological space, $S\subset O$, $x\in O$.  
We denote by $\Conn_O S$ and $\conn_O (S,x) \in \Conn_O S$ 
a set of all connected components of $S$ and  its connected component containing  $x$. We write $\clos_O S$, $\Int_O S$,  and $\partial_O S$ for  the \textit{closure,\/} the\textit{ interior,\/}  and the \textit{boundary\/} of $S$ in  $O$. 
The set $S$ is \textit{$O$-precompact\/} if $\clos_O S$ is a compact subset of $O$, and we write $S\Subset O$.  
\begin{definition}
\label{df:hole}
An arbitrary  $O$-precompact  connected component of $O\setminus S$ is called  a \textit{hole\/} in  $S$ with respect to\/ $O$. 
The union of a subset $K\subset O$ with all holes in it will be called an \textit{inward-filled hull\/} of this set $K$ with respect to $O$ and is denoted further as 
\begin{equation}\label{inhull}
\inhull_O K:=K\bigcup \Bigl(\bigcup \{C\in \Conn_O (O\setminus K) \colon C\Subset O
\}\Bigr). 
\end{equation}
Denote by $O_{\infty}$  the {\it  Alexandroff one-point compactification of\/} $O$ 
with underlying set $O \sqcup \{\infty\}$, where $\sqcup$ is the \textit{disjoint union\/} of $O$ with a single point $\infty \notin O$. If this space $O$ is a topological subspace of some ambient topological space $T\supset O$, then this point $\infty$  can be identified with the boundary $\partial O\subset T$ , considered as a single point $\{\partial O\}$.
\end{definition} 
Throughout this article, we use these topological concepts only in  cases when  $O$ is an {\it open non-empty proper Greenian open set} \cite[Ch.5, 2]{Helms} of $\RR_{\infty}^d=:T$, i.\,e.,   
\begin{subequations}\label{ODj}
\begin{align}
\varnothing \neq O=\Int_{\RR_{\infty}^d}O= \bigsqcup_{j\in N_O} D_j\neq \RR_{\infty}^d, \quad j\in N_O\subset \NN, \quad D_j=\conn_{\RR_{\infty}^d}(O,x_j), 
\tag{\ref{ODj}O}\label{{ODj}O}
\\
\intertext{where points  $x_j$  lie in {\it different connected components\/} $D_j$ of $O\subset \RR_{\infty}^d$;}
 \varnothing \neq D\neq \RR_{\infty}^d \quad \text{is an open connected subset, i.\,e., 
a {\it domain}}.
\tag{\ref{ODj}D}\label{{ODj}D}
\end{align}
\end{subequations}
For an open set $O$ from \eqref{{ODj}O}, we often use statements that are proved in our references only for domains $D$ from \eqref{{ODj}D}. This is acceptable since all such cases concern only to individual domains-components $D_j$. So, if $S\Subset O$, then $S$ meets only finite many components $D_j$. In addition, we give proofs of our statements only for cases $O,D \subset \RR^d$. If we have 
$o\notin D_j=D\ni \infty$, then we can to use the  \textit{inversion\/} relative to the 
sphere $\partial B(o,1)$ centered at $o \in  \RR^d$ as in \eqref{stK}.
\begin{theorem}[{\rm \cite[6.3]{Gardiner}, \cite{Gauther_B}}]\label{KOc}
Let  $K$ be a compact  set in an open set  $O \subset \RR^d$.  Then 
\begin{enumerate}[{\rm (i)}]
\item\label{Ki} $\inhull_{O}  K$ is a compact subset in $O$;
\item\label{Kii} the set\/ $O_{\infty} \setminus \inhull_{O}  K$ is connected and locally connected subset in\/ $O_{\infty}$; 
\item\label{Kiii}   the inward-filled hull of $K$ with respect to $O$ coincides with the complement in $O_{\infty}$ of  connected component of $O_{\infty}\setminus K$ containing the point $\infty$, i.\,e., 
\begin{equation*}
\inhull_{O}  K=O_{\infty}\setminus \conn_{O_{\infty}\setminus K}(\infty);
\end{equation*}
\item\label{Kiiv} if $O'\subset \RR_{\infty}^d$ is an open subset and 
 $O\subset  O'$ then 
 $\inhull_{O}  K\subset \inhull_{O'}  K$;
\item\label{Kiv} $\RR^d\setminus \inhull_{O}  K$ has only finitely many components, i.\,e., $$\#\Conn_{\RR^d_{\infty}}(\RR^d\setminus \inhull_{O}  K)<\infty.$$
\end{enumerate}
\end{theorem}

\section{Properties of  balayage of charges and measures}\label{Ssec_balm}
\setcounter{equation}{0}

In this section \ref{Ssec_balm} we discuss conventional classical balayage that is particular case of  \eqref{b0}.

\begin{definition}\label{df:1} Let $\vartheta, \mu  \in \Meas(S)$, $S\subset \Borel (\RR^d_{\infty})$.   Let $H\subset \RR_{\pm\infty}^{S}$ be a class of Borel-me\-a\-s\-u\-r\-a\-ble functions on $S$. 
 Let us assume that the integrals $\int h \dd{\vartheta}$ and  $\int h \dd{\mu}$ are well defined with values in $\RR_{\pm\infty}$ for each function $h\in H$. We write ${\vartheta} \preceq_H \mu$ and say that the charge  $\mu$ is a {\it balayage,\/} or, sweeping (out), of the charge ${\vartheta}$ {\it for\/} $H$, or, briefly, $\mu$ is a  $H$-balayage of $\vartheta$,   if 
\begin{equation}\label{balnumu}
\int h \dd {\vartheta} \overset{\eqref{b0mu}}{\leq} \int h\dd \mu \quad\text{\it for all\/ $h\in H$.}
\end{equation} 
\end{definition}

 In this article, we consider only balayage  for 
\begin{equation}\label{hHS}
\boxed{H\subset \sbh(S)\subset \RR_{-\infty}^S}.
\end{equation}
In this case,  the integrals from \eqref{balnumu} with values in $\RR_{-\infty}$ are well  defined for all measures 
${\vartheta}, \mu \in \Meas_{\comp}^+(S)$,
and, with values in $\RR$, for all  absolutely continuous
(with respect to $\lambda_d$) charges 
${\vartheta},\mu\overset{\eqref{MLl}}{\in} L_{\loc}^1(S)\dd \lambda_d$ etc.

\begin{theorem}\label{Prtr}  Let  $O\subset \RR^d$ be an open set, $\mu \in \Meas(O)$ be a  $H$-balayage of ${\vartheta}\in  \Meas(O)$, $O'\subset \RR^d$ be an open set, and $H'\subset \RR_{\pm\infty}^{O'}$. 
\begin{enumerate}[{\rm 1.}]
\item\label{b1} If  $1\in H$, then ${\vartheta} (O)\leq \mu(O)$. 
\item\label{b2}  If $\pm 1\in H$, then ${\vartheta} (O)= \mu(O)$.
\item\label{b3} If $H'\subset H$, then  $\mu$  is a $H'$-balayage of ${\vartheta}$.
 \item If  $O'\subset O$  and $supp \vartheta\cup \supp \mu\subset O'$, then $\mu\bigm|_{O'}$  is a balayage of ${\vartheta}\bigm|_{O'}$ for  $H\bigm|_{O'}$.
 \end{enumerate}
\end{theorem}
All statements of Theorem \ref{Prtr} are obvious.

\begin{remark} Balayage of  charges and measures with a non-compact support is also occur frequently and are used  in Analysis. So, a bounded domain $D\subset \RR^d$ is called a {\it quadrature domain\/} (for harmonic functions)  if there is a charge $\mu \in \Meas_{\comp} (D)$ such that the restriction $\lambda_d\bigm|_D$  is a balayage of $\mu$ for the class $\har (D)\cap L^1(D)$. In connection with the quadrature domains, see very informative overview  \cite[3]{quad} and bibliography in it.
\end{remark}

\begin{theorem}\label{pr:diff} If $\mu \in \Meas^+_{\comp}(O)$   is a balayage of $\vartheta\in \Meas_{\comp}^+(O)$ for $\bigl(\sbh(O)\cap C^{\infty}(O)\bigr)$, then  $\mu$   is a balayage of $\vartheta$ for $\sbh(O)$. 
 \end{theorem} 
Theorem \ref{pr:diff} follows from \cite[Ch. 4, 10, Approximation Theorem]{Doob}.

\begin{example}[{\rm \cite{Gamelin}, \cite{C-R}, \cite{C-RJ}, \cite{Schachermayer}}]\label{sbhJ} Let $x\in O$. 
If a measure $\mu\in \Meas_{\comp}^+(O)$  is  a balayage of the Dirac measure $\delta_{x}$ for $\sbh(O)$, then this measure $\mu$ is called a {\it Jensen measure for\/} $x$. The class of such measures is denoted by $J_x(O)$. 
\end{example}
\begin{example}\label{sbhom} 
 We denote by   $\omega_D\colon D\times{\Borel}(\partial D)\to [0,1]$ the \textit{harmonic measure for\/} $D$ with non-polar boundary $\partial D\subset \RR_{\infty}^d$. Measures $\omega_D(x,\cdot)\in \Meas^+(\partial D)$, $x\in D$, will also be called a harmonic measure  for $D$, but with specification,  \textit{at\/} $x\in D$.  If $D\Subset O$, then measures 
\begin{equation}\label{om}
a\delta_x+b\omega_D(x,\cdot)\in J_x(O), \quad a,b\in \RR^+, \quad a+b=1.
\end{equation}
Likewise, if \begin{equation*}
\sum_{k\in \NN}b_k=1,\quad  b_k \in \RR^+, \quad D_k, D:=\bigcup_kD_k\Subset O
\quad\text{ are Greenian, then $\sum_k b_k\omega_{D_k}(x,\cdot)\in J_x(O)$.}  
\end{equation*}

So, the surface measure  $\sigma_{d-1}$ in the unit sphere $\partial \BB$ belong to  $J_0(r\DD)$ for any $r>1$. 
\end{example}
\begin{example}\label{aJ} 
Useful examples of Jensen measures from $J_0(B(r))$ are \textit{probability\/} measures
\begin{equation}\label{eps}
 \alpha_r^{\infty}\overset{\eqref{MLl}}{\in} \left(C_0^{\infty}(r\BB)\right)^+\dd \lambda_d\in \Meas^{1+}\bigl(B(r)\bigr),\quad \alpha_r^{\infty}(S)=\alpha_1^{\infty}( S/r), \quad S\in \Borel(\RR^d), 
\end{equation}
$ r\in \RR_*^+$,
invariant under the action of the orthogonal group on $\RR^d$.
\end{example}
\begin{example}[{\rm \cite{Gamelin}, \cite{Kha03}}]\label{sbhAS} Let $x\in O$.   If $\mu\in \Meas_{\comp}^+(O)$ is a balayage of $\delta_{x}$ for $\Har(O)$, then such measure $\mu$ is called a {\it Arens\,--\,Singer  measure for\/} $x$. The class of such measures is denoted by $AS_x(O)\supset J_x(O)$. Arens\,--\,Singer measures are often referred to as representing measures.
\end{example}

\begin{theorem}\label{pr:ii} For  $H\overset{\eqref{hHS}}{\subset} \sbh(O)$ (resp.,
$H =-H\subset \har(O)$), let $\mu\in \Meas_{\comp}(O)$ be a  balayage of $\vartheta\in \Meas_{\comp}$ for $H$. Let $\iota_x\in J_x(O)\subset \Meas^{1+}_{\comp}(O) $ 
(resp., $\iota_x\in AS_x(O)\subset \Meas^{1+}_{\comp}(O) $)
with 
\begin{equation}\label{sxi}
s_x:=\supp \iota_x\Subset B\Bigl (x,\frac{1}{2}\dist (\supp \mu, \partial O)\Bigr) \quad\text{for all $x\in \supp \mu$}
\end{equation} 
be a family Jensen  (resp., Arens\,--\,Singer) measures for points $x\in O$. The measure $\mu$  and probability measures $\iota_x$ are bounded  in aggregate, and we can to define the integral  of $\iota_x$ with respect to $\mu$ {\rm \cite{Landkoff}, \cite{Bourbaki}}
\begin{equation}\label{iimu}
\beta:=\int \iota_x\dd\mu(x)\colon h\longmapsto \int\Bigl(\int h\dd\iota_x\Bigr ) \dd \mu(x)
\end{equation}   
In particular, if every Jensen (resp., Arens\,--\,Singer) measure $\iota_x$ is a parallel shift to a point $x\in O$ of the same Jensen (resp., Arens\,--\,Singer) measure $\iota_0 $ for $0$ with the diameter $\diam \supp \iota_0$ of $\supp \iota_0$ fewer than
$\frac{1}{2}\dist (\supp \mu, \partial O)\bigr)$, then  integral $\beta$ from \eqref{iimu} 
is a classical  convolution $\beta =\iota_0*\mu$ of two measures $\iota_0$ and $\mu$:
\begin{equation}\label{sv}
\beta:=\iota_0 \ast \mu =\mu\ast \iota_0 \colon h\longmapsto \iint h (x+y) \dd \iota_0 \dd \mu,
\quad s_x\overset{\eqref{sxi}}{=}x+\supp \iota_0.
\end{equation}
In these cases both  measures $\beta$ from \eqref{iimu}--\eqref{sv} also a balayage of $\vartheta$ for $H$  with 
\begin{equation*}
\supp \beta\subset  \clos \Bigl((\supp \mu) \bigcup \bigcup_{x\in \supp \mu}s_x\Bigr)
\Subset O. 
\end{equation*}
\end{theorem}
\begin{proof} Under condition \eqref{sxi}, for subharmonic function $h\in H\subset \sbh(O)$, we have
\begin{equation}\label{ins}
\int h\dd \vartheta \leq \int h\dd \mu \leq 
\int_{\supp \mu }\Bigl(\int_{s_x}h \dd \iota_x\Bigr)\dd \mu (x)
\overset{\eqref{iimu}}{=}\int h \dd \beta 
\end{equation}
by definitions  \eqref{iimu}--\eqref{sv}. For $H\subset \har (O)$ and $\iota_x\in AS_x$, by analogy with \eqref{ins}, we have equalities in \eqref{ins}.
\end{proof}
\begin{remark}\label{remgl}
If we choose parallel shifts to $x$ of measures (Example \ref{aJ}, \eqref{eps}) 
$$\alpha_{r(x)}^{\infty}
\in \left(C_0^{\infty}(r(x)\BB)\right)^+\dd \lambda_d\in \Meas^{1+}B(r(x))
$$ as measures $\iota_x$ for  Theorem  \ref{pr:ii} with a function $r\in C^{\infty}(O)$ and with condition  \eqref{sxi}, then our measures $\beta$ from  from \eqref{iimu}--\eqref{sv} both measures belong to the class $C_0^{\infty}(O)\dd \lambda_d$ and still $\vartheta \preceq_{H} \beta $, i.e., 
the measure $\beta \in \Meas_{\comp}(O) $ is a balayage of the measure $\vartheta$ for $H$.
\end{remark}

\begin{theorem}\label{Pr:munuh}  Let  $\mu \in \Meas_{\comp}(O)$ be a  balayage of ${\vartheta}\in  \Meas_{\comp}(O)$ for $\Har (O)$.  Then 
\begin{equation}\label{bh}
\int h \dd {\vartheta}=\int h\dd \mu \quad \text{for any $h \in  \Har \bigl({\inhull_O (\supp \mu \cup \supp {\vartheta}})\bigr)$}
\end{equation}
{\rm (see Subsec. \ref{hullin}, Definition \ref{df:hole} 
of inward-filled hull of compact subset $\supp \mu \cup \supp {\vartheta}$
 in $O$).}
\end{theorem}
\begin{proof}  We set 
\begin{equation}\label{K}
K:=\supp \vartheta \cup \supp \mu \Subset O.
\end{equation} 
By Theorem \ref{KOc} and \cite[Theorem 1.7]{Gardiner}, if $h\in \Har\bigl({\inhull_O K}\bigr)$, then there are functions $h_k\in \Har (O)$, $k\in \NN$, such that the sequence 
 $(h_k)_{k\in \NN}$ converges to $h$ in $C \bigl( {\inhull_O K}\bigr)$, and 
\begin{multline*}
\int_{{\inhull_O K}} h \dd {\vartheta}= \int_{{\inhull_O K}}  \lim_{k\to \infty}
h_k \dd {\vartheta} = \lim_{k\to \infty} \int_{O} h_k \dd {\vartheta}
\\
\overset{\eqref{balnumu}}{\leq} 
\lim_{k\to \infty} \int_{O} h_k \dd \mu= 
 \int_{{\inhull_O K}} \lim_{k\to \infty} h_k \dd \mu=
 \int_{{\inhull_O K}} h \dd {\vartheta}.
\end{multline*}
Using the opposite function $-h\in  \Har\bigl({\inhull_O K}\bigr)$, we have the inverse inequality.  
\end{proof}

\begin{theorem}\label{pr:4}  Let  $\mu \in \Meas_{\comp}(O)$ be a  balayage of ${\vartheta}\in  \Meas_{\comp}(O)$ for $\sbh (O)$.  Then 
\begin{equation}\label{bhs}
\int  u \dd {\vartheta}\leq \int  u\dd \mu \quad \text{ for all $u\in \sbh \bigl({\inhull_O (\supp \mu \cup \supp {\vartheta})}\bigr)$},
\end{equation}
i.\,e.,  if $O'\supset\inhull_O (\supp \mu \cup \supp {\vartheta})$ is an open set, then $\mu$ is a    $\sbh(O')$-balayage of $\vartheta$. 
\end{theorem}
\begin{proof} We use  the notation \eqref{K}. By Theorem \ref{KOc},  if $ u\in \sbh\bigl({\inhull_O K}\bigr)$, then there is  a function $ U \in \sbh \bigl(O)$ such that $u=U$ on  $ {\inhull_O K}$ \cite[Theorem 6.1]{Gardiner},
\cite[Theorem 1]{Gauther_C}, \cite[Theorem 16]{Gauther_B}, and we have
\begin{equation*}
\int_{{\inhull_O K}}  u \dd {\vartheta}= \int_{{\inhull_O K}}  
 u \dd {\vartheta} = \int_{O} U \dd {\vartheta}
\overset{\eqref{balnumu}}{\leq} 
 \int_{O} U \dd \mu= 
 \int_{{\inhull_O K}} U \dd \mu=
 \int_{{\inhull_O K}}  u \dd {\mu},
\end{equation*}
that   gives \eqref{bhs}.
\end{proof}

\begin{theorem}\label{Pr_pol}
If  $\mu\in \Meas_{\comp}^+(O)$ is  a\/  $\sbh(O)$-balayage of a measure ${\vartheta} \in \Meas_{\comp}^+(O)$, and  a set $E$ is polar, then  $\mu (O\cap E\setminus \supp {\vartheta})=0$.
\end{theorem}
\begin{proof} There is $k_0\in \NN$ such that $B(x,1/k_0)\Subset O$ for all 
$x\in \supp {\vartheta}$.  For any $k\in k_0+\NN_0$ 
there exists an finite cover of $\supp {\vartheta}$ by balls $B(x_j,1/k)\Subset O$ such that the open subsets 
\begin{equation*}
O_k:=\bigcup_j B(x_j,1/k)\Subset O,\quad
 \supp {\vartheta} \Subset O_k\supset O_{k+1}, \quad k\in k_0 +\NN_0, \quad \supp {\vartheta} =\bigcap_{k\in k_0+\NN_0} O_k,  
\end{equation*}
have complements $\RR_{\infty}^d \setminus O_k$ in $\RR_{\infty}^d$ \textit{without isolated points.\/} Then 
 every open set  $O_k\Subset O$ is regular for the Dirichlet problem.
It suffices to prove that the equality $\mu (\mathcal O_k\cap E)=0$ holds for every number  
 $k\in k_0+\NN_0$. By definition of polar sets, there is a  function $u\in \sbh_*(O)$ such that $u(E)=\{-\infty\}$. Consider the functions 
\begin{equation}\label{Uk}
U_k=\begin{cases}
u \text{ \it  on $O\setminus O_k$},\\
\text{\it the harmonic extension of $u$ from $\partial O_k$ into $O_k$}\text{ on $O_k$},
\end{cases}   
\qquad k\in k_0+\NN_0.
\end{equation}
We have  $U_k\in \sbh_*(O)$, and $U_k$ is bounded below in $\supp {\vartheta} \Subset O_k$. Hence
\begin{multline*}\label{<U}
-\infty <\int_O U_k \dd {\vartheta}
\overset{\eqref{balnumu}}{\leq}
\int_O U_k \dd \mu=
\left(\int_{O\setminus (O_k\cap E)}+\int_{O_k\cap E}\right) U_k \dd \mu
\\
=\int_{O\setminus (O_k\cap E)} U_k \dd \mu+(-\infty)\cdot \mu(O_k\cap E)
\overset{\eqref{{infty}0}}{\leq} \mu(O) \sup_{\supp \mu} U_k+(-\infty)\cdot \mu(O_k\cap E).
\end{multline*}
Thus, we have  $\mu(O_k\cap E)=0$.
\end{proof}

Generally speaking, Theorem \ref{Pr_pol} is not true for $\Har(O)$-balayage.  An implicit example built in \cite[Example]{MenKha19}. 
We will indicate in Example \ref{5} one more constructive and general way of building in this direction

\begin{example}[{\rm development of one example  of T. Lyons \cite[XIB2]{Koosis}}] \label{5} Consider 
\begin{equation}\label{Eas}
O=\BB,  \quad 0<r_0<r<1,  \quad \vartheta\overset{\eqref{df:spb}}{:=}
\frac1{b_dr_0^d}\lambda_d\bigm|_{r_0\BB}, 
\quad \mu \overset{\eqref{df:spb}}{:=}
\frac1{b_dr^d}\lambda_d\bigm|_{r\BB}  
\end{equation} 
Easy to see that $\vartheta \preceq_{\sbh(\BB)} \mu$. Let $E=(e_j)_{j\in \NN}\Subset r\BB\setminus r_0\overline \BB$ be a polar countable set without limit point in $r\BB\setminus r_0\overline \BB$.  Surround each point $e_j\in E$ with a ball $B(e_j,r_j)$ of such a small radius  $r_j>0$ that the union of all these balls is contained in
$r\BB\setminus r_0\overline \BB$. Consider a measure 
\begin{align*}
\mu_E&\overset{\eqref{mB}}{:=}\mu-\frac{1}{b_dr^d}\sum_{j\in \NN} \lambda_d\bigm|_{B(e_j,r_j)}+\frac{1}{b_dr^d}\sum  \lambda_d(e_j,r_j)\delta_{e_j}\\
&\overset{\eqref{Eas}}{=}
\frac1{b_dr^d}\lambda_d\bigm|_{r\BB}-\frac{1}{b_dr^d}\sum_{j\in \NN} \lambda\bigm|_{B(e_j,r_j)}+\frac{1}{r^d}\sum_{j\in\NN} r_j^d \delta_{e_j} .
\end{align*}
By construction, the measure $\mu_E$ is $\har(\BB)$-balayage of measure $\vartheta$, 
but \begin{equation*}
\mu_E(E)=\frac{1}{r^d}\sum_{j\in \NN} r_j^d >0.
\end{equation*}
in direct contrast to Theorem \ref{Pr_pol}. 
\end{example}

\end{document}